\documentclass{birkjour}
 \newtheorem{thm}{Theorem}[section]
 \newtheorem{cor}[thm]{Corollary}

 \theoremstyle{definition}
 \newtheorem{defn}[thm]{Definition}
 \theoremstyle{remark}
 \newtheorem{rem}[thm]{Remark}
 \newtheorem{ex}[thm]{Example}
 \numberwithin{equation}{section}
\usepackage{xcolor}
\usepackage[commentmarkup=footnote]{changes}

\begin{document}

%
%
%
%
%
%
%
%
%

\title[On the Interpolating Sesqui-Harmonicity of Vector Fields]
 {On the Interpolating Sesqui-Harmonicity of Vector Fields}

\author[Bouazza Kacimi]{Bouazza Kacimi}

\address{%
University of Mascara\\
Faculty of Exact Sciences\\
Department of Mathematics\\
BP 305 Route de Mamounia\\
29000 Mascara\\
Algeria}

\email{bouazza.kacimi@univ-mascara.dz}

\author{Amina Alem}
\address{%
University of Mascara\\
Faculty of Exact Sciences\\
Department of Mathematics\\
BP 305 Route de Mamounia\\
29000 Mascara\\
Algeria}

\email{amina.alem@univ-mascara.dz}

\author{Mustafa \"{O}zkan}
\address{ Gazi University \\
Faculty of Sciences\\
Department of Mathematics \\
06500 Ankara\\
Turkey}
\email{ozkanm@gazi.edu.tr}



\date{January 1, 2004}


\begin{abstract}
This article deals with the interpolating sesqui-harmonicity of a vector field $X$ viewed as a map from a Riemannian manifold $(M,g)$ to its tangent bundle $TM$ endowed with the Sasaki metric $g_{S}$. We show characterization theorem for $X$ to be interpolating sesqui-harmonic map. We give also the critical point condition which characterizes interpolating sesqui-harmonic vector fields. When $(M,g)$ is compact and oriented and under some conditions, we prove that $X$ is an interpolating sesqui-harmonic vector field (resp. interpolating sesqui-harmonic map) if and only if $X$ is parallel. Moreover, we extend this result for a left-invariant vector field on a Lie group $G$ having a discrete subgroup $\Gamma$ such that the quotient $\Gamma\backslash G$ is compact.
\end{abstract}
\maketitle
\section{Introduction}
The theory of harmonic maps is of essential importance in both geometry, analysis and physics. For example, on mathematical side harmonic maps are among the most studied variational problems in geometric analysis, and on physical side they are closely related to nonlinear field theory in theoretical physics. Given a smooth map
$\varphi: M\rightarrow N$ between two Riemannian manifolds $(M,g)$ and $(N,h)$ of dimensions $m$ and $n$ respectively, the map $\varphi$ is said to be harmonic if it is a critical
point of the energy functional defined by:
\begin{equation}\label{eq1.1}
    E(\varphi)=\frac{1}{2}\int_{D}\left\|d\varphi\right\|^{2}v_{g},
\end{equation}
for any compact domain $D$ of $M$, where $v_{g}$ is the volume element of $(M,g)$. The Euler-Lagrange equation of $E(\varphi)$ is \cite{baird,eells2}
\begin{equation*}
    \tau(\varphi)=\operatorname{Tr}_{g}(\nabla d\varphi)=\sum_{i=1}^{m}
    \{\nabla^{\varphi}_{e_{i}}d\varphi(e_{i})-d\varphi(\nabla_{e_{i}}e_{i})\}=0,
\end{equation*}
where $\tau(\varphi)$ is the tension field of $\varphi$, here $\nabla^{\varphi}$ is the connection on the vector bundle $\varphi^{-1}TN$ induced from the Levi-Civita connection $\nabla^{N}$ of $(N,h)$, $\nabla$ is the Levi-Civita connection of $(M,g)$ and $\{e_{i}\}_{i=1}^{m}$ is a
local orthonormal frame field of $(M,g)$. For a recent survey on harmonic maps see \cite{helein}. Among the first generalizations of harmonic maps is the notion of polyharmonic maps of order $k$ between Riemannian manifolds introduced by Eells and Lemaire in \cite{eells1}. For $k=2$, they defined the bienergy of $\varphi$ as the functional
\begin{equation}\label{eq1.2}
    E_{2}(\varphi)=\frac{1}{2}\int_{D}\left\|\tau(\varphi)\right\|^{2}v_{g},
\end{equation}
for any compact domain $D$ of $M$. The map $\varphi$ is said to be biharmonic if it is a critical point of the bienergy functional \eqref{eq1.2}. The associated Euler-Lagrange equation is derived by Jiang \cite{jiang} as follows:
\begin{equation*}
    \tau_{2}(\varphi) = \Delta^{\varphi}\tau(\varphi)-
    \overset{m}{\underset{i=1}{\sum}}R^{N}(\tau(\varphi),d\varphi(e_{i}))d\varphi(e_{i})=0,
\end{equation*}
where $\tau_{2}(\varphi)$ is the bitension field of $\varphi$,
$\Delta^{\varphi}\tau(\varphi)=-\sum_{i=1}^{m}
(\nabla^{\varphi}_{e_{i}}\nabla^{\varphi}_{e_{i}}\tau(\varphi)-
\nabla^{\varphi}_{\nabla^{M}_{e_{i}}e_{i}}\tau(\varphi))$ is the rough Laplacian on $\varphi^{-1}TN$,
and $R^{N}$ is the curvature tensor of the target manifold $N$. By definition, it can be seen that every harmonic map is biharmonic. However, a biharmonic map can be non-harmonic in which case it is called proper biharmonic. Biharmonic maps have been studied by several researchers see for instance \cite{ou1,urakawa} and references therein.

Branding \cite{branding1}, introduced an action functional for maps between Riemannian manifolds that interpolated between the actions for harmonic and biharmonic maps as follows:
 \begin{equation}\label{eq1.3}
    E_{\delta_{1},\delta_{2}}(\varphi)=\delta_{1}\int_{D}\left\|d\varphi\right\|^{2}v_{g}+
    \delta_{2}\int_{D}\left\|\tau(\varphi)\right\|^{2}v_{g},
\end{equation}
where $\delta_{1},\delta_{2}\in\mathbb{R}$. The functional \eqref{eq1.3} is
used in string theory of physics, it is known as bosonic string with extrinsic curvature term, see \cite{polyakov}. The map $\varphi$ is said to be interpolating sesqui-harmonic if it is a critical point of the functional \eqref{eq1.3}.
The associated Euler-Lagrange equation is \cite{branding1}
\begin{equation}\label{eq1.4}
    \tau_{\delta_{1},\delta_{2}}(\varphi) =\delta_{1}\tau(\varphi)+ \delta_{2}\tau_{2}(\varphi)=0.
\end{equation}
We mention that the field $\tau_{\delta_{1},\delta_{2}}(\varphi)$ is appeared with different sign from the expression in \cite{branding1} due to the different sign of the bitension field. It is obvious that harmonic maps solve \eqref{eq1.4}. Moreover, if $\delta_{1}=0$ and $\delta_{2}=1$, then an interpolating sesqui-harmonic map turns into a biharmonic map. So, harmonic maps and biharmonic maps are subclass of interpolating sesqui-harmonic maps. Note that there have been several articles dealing with particular aspects of \eqref{eq1.3} see \cite{branding1}. In \cite{branding2}, Branding studied various analytic aspects of interpolating sesqui-harmonic maps between Riemannian manifolds particulary a spherical target. In \cite{karaca}, Karaca et al. considered
interpolating sesqui-harmonic curves in Sasakian space forms and obtained the necessary and sufficient conditions for Legendre curves to be interpolating sesqui-harmonic.

On the other hand, denote by $\mathfrak{X}(M)$ the set of all smooth vector
fields on $M$ and by $g_{S}$ the Sasaki metric on the tangent bundle $TM$. Any $X\in\mathfrak{X}(M)$ defines a
smooth map from $(M,g)$ to $(TM,g_{S})$. When $M$ is compact, it was proved in \cite{ishihara} and also in \cite{nouhaud} that $X:(M, g)\rightarrow (TM,g_{S})$ is an harmonic map if and only if $X$ is parallel; moreover this result remains true if $X$ is a harmonic section of $TM$, i.e., a critical point of the energy functional $E$ restricted to the set $\mathfrak{X}(M)$  \cite{medrano}. The bienergy of $X\in\mathfrak{X}(M)$ is the bienergy of the corresponding map see \cite{markellos1}. In \cite{markellos1}, the authors obtained the critical point of the bienergy functional $E_{2}$ restricted to the set $\mathfrak{X}(M)$ (equivalently, $X$ is a biharmonic vector field) and showed that if $M$ is compact, then $X$ is biharmonic vector field (resp. biharmonic map) if and only if $X$ is parallel. In \cite{alem}, the authors established the formula of the bitension field of $X:(M, g)\rightarrow (TM,g_{S})$, and obtained characterization theorem for $X$ to be biharmonic map, furthermore they gave the relationship between the notion of biharmonic vector field and that of vector field which is biharmonic map. Especially, they proved that a  left-invariant vector field $X$ on three dimensional unimodular Lie group is biharmonic vector field (resp. biharmonic map) if and only if $X$ is parallel.

In this note, we will study the interpolating sesqui-harmonicity of $X\in\mathfrak{X}(M)$ viewed as a map  $X:(M, g)\rightarrow (TM,g_{S})$ building on the developed results for harmonic vector fields \cite{dragomir,medrano}, vector fields which are harmonic maps \cite{ishihara,nouhaud}, biharmonic vector fields \cite{markellos1} and vector fields which are biharmonic maps \cite{alem}. So, it is natural to address
the problem of characterizing those vector fields for which the corresponding
map is an interpolating sesqui-harmonic map, and also to check whether vector fields $X$ that are
critical points of the functional $E_{\delta_{1},\delta_{2}}$ restricted to variations through vector fields, such a vector field is called interpolating sesqui-harmonic vector field.

This paper is organized as follows. Section $2$ contains some basic notions, properties and results that will be needed later. In section $3$, we establish the formula of the field $\tau_{\delta_{1},\delta_{2}}(X)$ of $X:(M, g)\rightarrow (TM,g_{S})$ (see Theorem \ref{Th1}) and provide the conditions which characterize vector fields which are interpolating sesqui-harmonic maps (see Theorem \ref{Th2}). By means of the formula of the field $\tau_{\delta_{1},\delta_{2}}(X)$ of $X$, we derive the first variational formula associated to the energy functional $E_{\delta_{1},\delta_{2}}$ restricted to the space $\mathfrak{X}(M)$ (see Theorem \ref{Th3}). As a corollary, we obtain the critical point condition characterizes interpolating sesqui-harmonic vector field (see Corollary \ref{co1}). Consequently, we obtain the relationship between interpolating sesqui-harmonic vector fields and vector fields which are interpolating sesqui-harmonic maps (see Corollary \ref{co2}). Afterwards, if $\delta_{1}$ and $\delta_{2}$ have the same sign, and $(M,g)$ is compact and oriented, we prove that $X$ is an interpolating sesqui-harmonic vector field (resp. interpolating sesqui-harmonic map) if and only if $X$ is parallel (see Theorem \ref{Th4} and Theorem \ref{Th5}), we present an example of non-parallel vector field which is an interpolating sesqui-harmonic map on the Sol space (see Example \ref{Ex1}). Finally, in section $4$, if $\delta_{1}$ and $\delta_{2}$ have the same sign, and $G$ is a Lie group having a discrete subgroup $\Gamma$ such that the quotient $\Gamma\backslash G$ is compact, we prove that a left-invariant vector field $X$ on $G$ is interpolating sesqui-harmonic (resp. interpolating sesqui-harmonic map) if and only if $X$ is parallel (see Theorem \ref{Th6}). Thereby, when $\delta_{1}$ and $\delta_{2}$ have different signs, we completely determine the set of left-invariant interpolating sesqui-harmonic vector fields and left-invariant vector fields which are interpolating sesqui-harmonic maps on the Heisenberg group $Nil$ (see Theorem \ref{Th7}).

\section{Preliminaries}
We recall here some basic facts on the geometry of tangent bundle. We refer the reader to \cite{dombrowski,kowalski,yano} and references therein for further details. Let $(M,g)$ be a Riemannian manifold of dimension $m$ and $(TM, \pi, M)$ be its tangent
bundle, a local chart $(U, x^{i})_{1\leq i\leq m}$ on $M$ induces a local chart $(\pi^{-1}(U),x^{i},y^{i})_{1\leq
i\leq m}$ on $TM$. Denotes the Christoffel symbols of $g$ by $\Gamma_{jk}^{i}$, the tangent space $T_{(x,u)}TM$ at a point $(x,u)$ in $TM$ is a direct sum of the vertical subspace $\mathcal{V}_{(x,u)}=\ker(d\pi\mid_{(x,u)})$ and the horizontal
subspace $\mathcal{H}_{(x,u)}$, with respect to the Levi-Civita connection $\nabla $ of $M$:
\begin{equation*}
    T_{(x,u)}TM=\mathcal{H}_{(x,u)}\oplus\mathcal{V}_{(x,u)}.
\end{equation*}
Let $X|_{U} = X^{i}\frac{\partial}{\partial x^{i}}$ be a local vector field on $M$. The vertical and
the horizontal lifts of $X$ are defined respectively by:
\begin{equation*}
X^{v}|_{\pi^{-1}(U)}=(X^{i}\circ \pi)\frac{\partial}{\partial y^{i}},
\end{equation*}
and
\begin{equation*}
X^{h}|_{\pi^{-1}(U)}=(X^{i}\circ \pi)\frac{\partial}{\partial x^{i}}%
-(\Gamma_{jk}^{i}\circ \pi) (X^{j}\circ \pi)y^{k}\frac{\partial}{\partial
y^{i}}.
\end{equation*}
The Sasaki metric on $TM$ is the Riemannian metric $g_{S}$ defined by
\begin{equation}\label{eq2.1}
g_{S}(X^{h},Y^{h})=g_{S}(X^{v},Y^{v})=g(X,Y)\circ \pi, \;
g_{S}(X^{v},Y^{h})=0,
\end{equation}
for any $X,Y\in\Gamma(TM)$.

A vector field $X$ on $(M,g)$ can be viewed as the immersion
$X:(M,g)\rightarrow (TM,g_{S}): x\mapsto (x,X_{x})\in TM$ into its tangent bundle $TM$ equipped with the
Sasaki metric $g_{S}$. If $Y\in\Gamma(TM)$ then, we have (see \cite[pp. 50]{dragomir})
\begin{equation}\label{eq2.2}
dX(Y)=\{Y^{h}+(\nabla_{Y}X)^{v}\}\circ X.
\end{equation}
The tension field $\tau(X)$ is given by \cite{medrano} and rewritten in \cite{markellos1} as follows:
 \begin{equation}\label{eq2.3}
    \tau(X)=(-S(X))^{h}+( -\bar{\Delta}X)^{v},
\end{equation}
where $S(X)=\overset{m}{\underset{i=1}{\sum}}R(\nabla_{e_{i}}X,X)e_{i}$, here $R$ denotes the Riemannian curvature tensor taken with the sign convention
\begin{equation*}
    R(X,Y)Z=\nabla_{X}\nabla_{Y}Z-\nabla_{Y}\nabla_{X}Z-\nabla_{[X,Y]}Z,
\end{equation*}
for all vector fields $X$, $Y$ and $Z$ on $M$, and $\bar{\Delta}X$ is the rough Laplacian given by $\bar{\Delta}X=-tr_{g}(\nabla^{2}X)=
\overset{m}{\underset{i=1}{\sum}}(\nabla_{\nabla_{e_{i}}e_{i}}X-\nabla_{e_{i}}\nabla_{e_{i}}X).$ A vector field $X$ defines a harmonic map from $(M,g)$ to $(TM,g_{S} )$ if and only if $\tau(X)=0$, equivalently $\bar{\Delta}X=0$ and $S(X)=0$, $X$ is called harmonic vector field if it is a critical point of the energy functional \eqref{eq1.1}, only considering variations through vector fields. The corresponding Euler-Lagrange equation is given by $\bar{\Delta}X=0,$ so $X$ is a harmonic map if and only if $X$ is a harmonic vector field
and $S(X)=0$.
\section{The interpolating sesqui-harmonicity of vector fields}
In the sequel, we give the formula of the field $\tau_{\delta_{1},\delta_{2}}(X)$ of $X:(M, g)\rightarrow (TM,g_{S})$.
\begin{thm}\label{Th1}
Let $(M, g)$ be an $m$-dimensional Riemannian manifold and $(TM,g_{S})$ its tangent bundle equipped
with the Sasaki metric, if $X:(M, g)\rightarrow (TM,g_{S})$ is a smooth vector field then the field $\tau_{\delta_{1},\delta_{2}}(X)$ of $X$ is given by
\begin{align}\label{eq3.1}
   \tau_{\delta_{1},\delta_{2}}&(X)=
   \Big\{-\delta_{1}\bar{\Delta}X-\delta_{2}\bar{\Delta}\bar{\Delta}X-\delta_{2}
   \sum_{i=1}^{m}[(\nabla_{e_i}R)(e_i,S(X))X
   \nonumber\\
   &+R(e_i,\nabla_{e_i}S(X))X+2R(e_i,S(X))\nabla_{e_i}X]\Big\}^{v}+
   \Big\{-\delta_{1}S(X)-\delta_{2}\bar{\Delta}S(X)\nonumber\\&-
   \delta_{2}R(X,\bar{\Delta}X)S(X)+\delta_{2}\sum_{i=1}^{m}[R(X,\nabla_{e_i}\bar{\Delta}X)e_i
   -R(\nabla_{e_i}X,\bar{\Delta}X)e_i
   \nonumber\\
   &-R(e_i,S(X))e_i-(\nabla_{S(X)}R)(\nabla_{e_i}X,X)e_i+R(X,\nabla_{e_i}X)\nabla_{e_i}S(X)\nonumber\\
   &-  R(X,R(e_i,S(X))X)e_i]\Big\}^{h}.
\end{align}
\end{thm}
\begin{proof}On making use of \eqref{eq1.4} and \eqref{eq2.3} and Theorem 3.1 in \cite{alem}, we find easily the formula \eqref{eq3.1}.
\end{proof}
Then, we deduce
\begin{thm}\label{Th2}
Let $(M, g)$ be an $m$-dimensional Riemannian manifold and $X\in\mathfrak{X}(M)$, then $X:(M, g)\rightarrow (TM,g_{S})$ is an interpolating sesqui-harmonic map if and only if
\begin{multline*}
\delta_{1}\bar{\Delta}X+\delta_{2}\bar{\Delta}\bar{\Delta}X+\delta_{2}
\sum_{i=1}^{m}[(\nabla_{e_i}R)(e_i,S(X))X \\
+R(e_i,\nabla_{e_i}S(X))X+2R(e_i,S(X))\nabla_{e_i}X]=0,
\end{multline*}
and
\begin{multline*}
   \delta_{1}S(X)+\delta_{2}\bar{\Delta}S(X)+\delta_{2}R(X,\bar{\Delta}X)S(X)-\delta_{2}
   \sum_{i=1}^{m}[R(X,\nabla_{e_i}\bar{\Delta}X)e_i \\-R(\nabla_{e_i}X,\bar{\Delta}X)e_i-R(e_i,S(X))e_i
   -(\nabla_{S(X)}R)(\nabla_{e_i}X,X)e_i \\+R(X,\nabla_{e_i}X)\nabla_{e_i}S(X)-
   R(X,R(e_i,S(X))X)e_i]=0,\nonumber
\end{multline*}
where $\{e_{i}\}_{i=1}^{m}$ is a local orthonormal frame field of $(M,g)$.
\end{thm}

Let $(M,g)$ be a compact
$m$-dimensional Riemannian manifold. The energy functional $E_{\delta_{1},\delta_{2}}$ of $X$ is
defined to be the energy of the corresponding map $X: (M,g)\longrightarrow (TM,g_{S})$. More
precisely, combining relations \eqref{eq2.1}, \eqref{eq2.2} and \eqref{eq2.3}, one obtain
\begin{align*}
    E_{\delta_{1},\delta_{2}}(X)&=\delta_{1}\int_{M}\left\|dX\right\|^{2}v_{g}+
    \delta_{2}\int_{M}\left\|\tau(X)\right\|^{2}v_{g},\nonumber\\
   &=\delta_{1}m\;Vol(M)+\delta_{1}\int_M\left\|\nabla X\right\|^{2}v_g+\delta_{2}\int_M\big[\left\|S(X) \right\|^{2}+\left\|\bar{\Delta}X\right\|^{2}\big] v_g.\nonumber
\end{align*}
\begin{defn}
Let $(M,g)$ be a Riemannian manifold. A vector field $X\in\mathfrak{X}(M)$ is called interpolating sesqui-harmonic
if the corresponding map $X: (M,g)\longrightarrow (TM,g_{S})$ is a critical point for the
functional $E_{\delta_{1},\delta_{2}}$, only considering variations through vector fields.
\end{defn}
In what follows, we determine the first variational
formula of the functional $E_{\delta_{1},\delta_{2}}$ restricted to the space $\mathfrak{X}(M)$. We prove the following theorem:
\begin{thm}\label{Th3}
Let $(M, g)$ be a compact oriented $m$-dimensional Riemannian manifold, $\{e_{i}\}_{i=1}^{m}$ a local orthonormal
frame field of $(M, g)$, $X$ a tangent vector field on
$M$ and $E_{\delta_{1},\delta_{2}}:\mathfrak{X}(M)\longrightarrow [0,+\infty)$ the functional $E_{\delta_{1},\delta_{2}}$ restricted to the space
of all vector fields. Then
\begin{align}\label{S1}
    \frac{d}{dt} E_{\delta_{1},\delta_{2}}(X_{t})\bigg|_{t=0}=&\int_{M} \Big\{g(2\delta_{1}\bar{\Delta}X+2\delta_{2}\bar{\Delta}\bar{\Delta}X+2
    \delta_{2}\sum_{i=1}^{m}\big[(\nabla_{e_i}R)(e_i,S(X))X\nonumber\\
   &+R(e_i,\nabla_{e_i}S(X))X
   +2R(e_i,S(X))\nabla_{e_i}X\big],V)\Big\}v_{g},
\end{align}
for any smooth $1$-parameter variation $U:M\times (-\epsilon,\epsilon)\rightarrow TM$ of $X$ through vector fields i.e., $X_{t}(z)=U(z,t)\in T_{z}M$ for any $|t|<\epsilon$ and $z \in M$, or equivalently $X_{t}\in \mathfrak{X}(M)$
for any $|t|<\epsilon$. Also, $V$ is the tangent vector field on $M$ given by
\begin{equation*}
V(z)=\frac{d}{dt}X_{z}(0) ,   \quad z\in M,
\end{equation*}
where $X_{z}(t)=U(z,t), \,(z,t)\in M\times (-\epsilon,\epsilon)$.
\end{thm}
\begin{proof}
Let $U:M\times (-\epsilon,\epsilon)\rightarrow TM$ be a smooth variation of $X$( i.e., $U(z,0)=X(z)$ for any
$z\in M$) such that $X_{t}(z)=U(z,t)\in T_{z}M$ for any $z\in M$ and any $|t|<\epsilon$. We have
\begin{equation*}
  E_{\delta_{1},\delta_{2}}(X_{t})= \delta_{1}\int_{M} \left\|dX_{t}\right\|^{2}v_{g}+\delta_{2}\int_{M} \left\|\tau(X_{t})\right\|^{2}v_{g}.
\end{equation*}
And, from \cite{branding1} one has
\begin{equation*}
  \frac{d}{dt} E_{\delta_{1},\delta_{2}}(X_{t})\bigg|_{t=0}= -2\int_{M} g_{S}(\mathcal{V},\tau_{\delta_{1},\delta_{2}}(X))v_{g},
\end{equation*}
where $\mathcal{V}(z)=\frac{d X_{t}(z)}{d t}\big|_{t=0},$
and from \cite[pp. 58]{dragomir}, one obtain
\begin{equation}\label{eq3.4}
  \mathcal{V}=V^{v}\circ X.
\end{equation}
By virtue of \eqref{eq3.4} and the formula of $\tau_{\delta_{1},\delta_{2}}(X)$ given by \eqref{eq3.1}, we get
\begin{align}\label{a}
    \frac{d}{dt} E_{\delta_{1},\delta_{2}}(X_{t})\bigg|_{t=0}=&-2\int_{M} g_{S}(V^{v},\tau_{\delta_{1},\delta_{2}}(X))v_{g}, \nonumber\\
    =&2\int_{M} \Big\{g(V,\delta_{1}\bar{\Delta}X+\delta_{2}\bar{\Delta}\bar{\Delta}X+
    \delta_{2}\sum_{i=1}^{m}\big[(\nabla_{e_i}R)(e_i,S(X))X\nonumber\\
   &+R(e_i,\nabla_{e_i}S(X))X
   +2R(e_i,S(X))\nabla_{e_i}X\big])\Big\}v_{g}.\nonumber
\end{align}
Then, the desired formula follows.
\end{proof}
\begin{rem}
As usual, Theorem \ref{Th3} holds when $(M,g)$ is a non-compact Riemannian manifold. Indeed, if $M$ is non-compact, we can take an open subset $W$ in $M$ whose closure is compact, and take an arbitrary $V$ but, the support of $V$ is contained in $W$ (see Proposition 3 of \cite{medrano}).
\end{rem}
Then, we deduce the following.
\begin{cor}\label{co1}
A vector field $X$ of an $m$-dimensional Riemannian manifold $(M,g)$ is interpolating sesqui-harmonic if and only if
\begin{multline*}
\delta_{1}\bar{\Delta}X+\delta_{2}\bar{\Delta}\bar{\Delta}X+\delta_{2}
\sum_{i=1}^{m}[(\nabla_{e_i}R)(e_i,S(X))X\\+R(e_i,\nabla_{e_i}S(X))X+2R(e_i,S(X))\nabla_{e_i}X]=0,
\end{multline*}
where $\{e_{i}\}_{i=1}^{m}$ is a local orthonormal frame field of $(M,g)$.
\end{cor}

A reformulation of Theorem \ref{Th2} is then
\begin{cor}\label{co2}
Let $(M,g)$ be an $m$-dimensional Riemannian manifold and $X\in\mathfrak{X}(M)$. Then
$X$ is an interpolating sesqui-harmonic map of $(M,g)$ to $(TM,g_{S})$ if and only if  $X$ is interpolating sesqui-harmonic vector field and
\begin{multline*}
   \delta_{1}S(X)+\delta_{2}\bar{\Delta}S(X)+\delta_{2}R(X,\bar{\Delta}X)S(X)-\delta_{2}
   \sum_{i=1}^{m}[R(X,\nabla_{e_i}\bar{\Delta}X)e_i\\-R(\nabla_{e_i}X,\bar{\Delta}X)e_i-R(e_i,S(X))e_i
   -(\nabla_{S(X)}R)(\nabla_{e_i}X,X)e_i\\+R(X,\nabla_{e_i}X)\nabla_{e_i}S(X)-
   R(X,R(e_i,S(X))X)e_i]=0.
\end{multline*}
\end{cor}
In the following Theorem, when $\delta_{1}$ and $\delta_{2}$ have the same sign, we study the condition under of which a vector field $X$ of a Riemannian manifold $(M,g)$ is interpolating sesqui-harmonic under the assumption that the base manifold $(M,g)$ is compact. In particular, we have
\begin{thm}\label{Th4}
Let $(M,g)$ be a compact oriented $m$-dimensional Riemannian manifold and $X\in\mathfrak{X}(M)$ a vector field. Assuming that $\delta_{1}$ and $\delta_{2}$ have the same sign. Then, $X: (M,g)\longrightarrow (TM,g_{S})$ is  interpolating sesqui-harmonic vector field if and only if $X$ is parallel.
\end{thm}
\begin{proof} We assume that $X$ is an interpolating sesqui-harmonic vector field i.e. critical point of the energy functional $E_{\delta_{1},\delta_{2}}$ restricted to the space $\mathfrak{X}(M)$. We consider the smooth 1-parameter variation $X_{t}=(1+t)X$ of $X$ ($|t|<\epsilon$). By the Lemma 2.15 in \cite{dragomir} we have
\begin{equation*}
    g(\bar{\Delta}X,X)=\frac{1}{2}\Delta(\left\|X\right\|^{2})+\left\|\nabla X\right\|^{2},
\end{equation*}
where $\Delta$ is the Laplace-Betrami operator on functions. Applying the divergence Theorem for the function $\left\|X\right\|^{2}$, we get
\begin{equation}\label{a4}
    \int_{M}g(\bar{\Delta}X,X)v_{g}=\int_{M}\left\|\nabla X\right\|^{2}v_{g}.
\end{equation}
Thus, using \eqref{S1} and \eqref{a4} and the proof of Theorem 3.4 in \cite{markellos1} we deduce that
\begin{align}\label{a5}
    0&=\frac{d}{dt} E_{\delta_{1},\delta_{2}}(X_{t})\bigg|_{t=0}=
    2\delta_{1}\int_{M}g(\bar{\Delta}X,X)v_{g}+2\delta_{2}\int_{M}g(\bar{\Delta}X,\bar{\Delta}X)v_{g}
    \nonumber\\&\quad +4\delta_{2}\int_{M}g(S(X),S(X))v_{g}\nonumber\\
    &=2\delta_{1}\int_{M} \left\|\nabla X\right\|^{2}v_{g}+2\delta_{2}\int_{M}\left\|\bar{\Delta}X\right\|^{2}v_{g}
    +4\delta_{2}\int_{M}\left\|S(X)\right\|^{2}v_{g}\nonumber.
\end{align}
Since $\delta_{1}$ and $\delta_{2}$ have the same sign and both functions $\left\|\nabla X\right\|^{2}$, $\left\|\bar{\Delta}X\right\|^{2}$,
$\left\|S(X)\right\|^{2}$ are positive, we conclude that $\nabla X=0$ i.e. $X$ is parallel. Conversely, we assume that the vector field $X$ is parallel, by virtue of Corollary \ref{co1}, $X$ is an interpolating sesqui-harmonic vector field.
\end{proof}
\begin{thm}\label{Th5}
Let $(M,g)$ be a compact oriented $m$-dimensional Riemannian manifold and $X\in\mathfrak{X}(M)$ a vector field. Assuming that $\delta_{1}$ and $\delta_{2}$ have the same sign, then $X: (M,g)\longrightarrow (TM,g_{S})$ is an interpolating sesqui-harmonic map if and only if $X$ is parallel.
\end{thm}
\begin{proof}
We assume that $X: (M,g)\longrightarrow (TM,g_{S})$ is an interpolating sesqui-harmonic map, then from Corollary \ref{co1}, $X$ is an interpolating sesqui-harmonic vector field, hence by Theorem \ref{Th4} $X$ is parallel. Conversely, we suppose that the vector field $X$ is parallel, by virtue of Theorem \ref{Th2}, $X$ is an interpolating sesqui-harmonic map.
\end{proof}
\begin{ex}\label{Ex1}
Consider the Sol space as the Cartesian $3$-space $\mathbb{R}^{3}(x,y,z)$ equipped with the metric $g_{Sol}=e^{2z}(dx)^{2}+e^{-2z}(dy)^{2}+(dz)^{2}$ and the orthonormal frame field
$\{e_{1}=e^{-z}\frac{\partial}{\partial x},e_{2}=e^{z}\frac{\partial}{\partial y},e_{3}=\frac{\partial}{\partial z}\}.$ We consider the vector field $X=f(z)e_{3}$, where $f(z)$ is a smooth real function depending of the variable
$z$. Drawing on computations from \cite{alem}, we obtain that $X=f(z)e_{3}$ is an interpolating sesqui-harmonic map if and only if the function $f$ satisfies the following homogeneous fourth order differential equation. 
\begin{equation}\label{eq3.39}
    \delta_{2}f''''-(\delta_{1}+4\delta_{2})f''+(2\delta_{1}+4\delta_{2})f=0.
\end{equation}
If $\frac{\delta_{1}+2\delta_{2}}{\delta_{2}}>0$, the general solution of \eqref{eq3.39} is 
\begin{equation}\label{eq3.40}
    f(z)=c_{1}e^{\sqrt{2}z}+c_{2}e^{-\sqrt{2}z}+
    c_{3}e^{\sqrt{\frac{\delta_{1}+2\delta_{2}}{\delta_{2}}}z}
    +c_{4}e^{-\sqrt{\frac{\delta_{1}+2\delta_{2}}{\delta_{2}}}z},
\end{equation}
where $c_{1},c_{2},c_{3}$ and $c_{4}$ are real constants. In particular, $X=f(z)e_{3}$ is also interpolating sesqui-harmonic vector
field, where $f(z)$ is given by \eqref{eq3.40}.
\end{ex}
\section{Interpolating sesqui-harmonicity of vector fields on Lie groups}
In this section, we investigate the interpolating sesqui-harmonicity of left-invariant vector fields on Lie groups.
We know that the action of any discrete subgroup $\Gamma$ of a Lie group $G$ by left translations is free
and properly discontinuous. Consequently, the set of orbits, that is, the space of right cosets $\Gamma\backslash G$,  is a $C^{\infty}$-manifold and the naturel projection $\pi:G\longrightarrow \Gamma\backslash G$, applying each $x$ to its orbit $\Gamma x$, is a $C^{\infty}$-mapping (see \cite{boothby}). Furthermore, each left-invariant vector field on $G$ descends to $\Gamma\backslash G$, or equivalently, if $X$ is left-invariant, then $\pi_{\ast}X_{ba}=\pi_{\ast}X_{a}$, for all $a\in G$ and $b\in \Gamma$, we have also each
left-invariant metric on $G$ and, in general, all its left-invariant tensor fields, descend to the quotient space $\Gamma\backslash G$ (see \cite{gonzalez}). Thereby, $\Gamma\backslash G$ is a Riemannian manifold with the same curvature properties for the curvature tensor as on $G$ (see \cite{gonzalez}). Thus, we deduce that the projections of left-invariant vector fields preserve the properties to be harmonic, biharmonic, interpolating sesqui-harmonic and to determine harmonic maps, biharmonic maps and interpolating sesqui-harmonic maps.

Now, we focus on the case of compact $\Gamma\backslash G$. It should be noted that a necessary and sufficient
condition for compactness is the existence of a compact subset $K\subset G$ whose $\Gamma$-orbits cover $G$, that is, $\Gamma K=G$ (see \cite{gonzalez}). In particular, any three-dimensional unimodular Lie group $G$ admits a discrete subgroup $\Gamma$ such that $\Gamma\backslash G$ is compact and conversely (see \cite{milnor}).

\begin{thm}\label{Th6}
Let $G$ be a Lie group having a discrete subgroup $\Gamma$ such that $\Gamma\backslash G$ is compact. Assuming that $\delta_{1}$ and $\delta_{2}$ have the same sign, then a left-invariant vector field $X$ on $G$ is interpolating sesqui-harmonic (resp. interpolating sesqui-harmonic map) if and only if $X$ is parallel.
\end{thm}
\begin{proof}
Suppose that $X$ is a left-invariant vector field on $G$ which is non-parallel interpolating sesqui-harmonic, it follows that its projection on
$\Gamma\backslash G$ which is also denoted by $X$ is non-parallel interpolating sesqui-harmonic and as $\Gamma\backslash G$ is compact, this gives a contradiction by combining with Theorem \ref{Th4}, so we must have $X$ is parallel. Using the similar way, we can prove the result when $X$ is an interpolating sesqui-harmonic map by virtue of Theorem \ref{Th5}.
\end{proof}
\begin{cor}
Let $G$ be a Lie group having a discrete subgroup $\Gamma$ such that $\Gamma\backslash G$ is compact. Then
\begin{enumerate}
  \item A left-invariant vector field $X$ on $G$ is harmonic (resp. harmonic map) if and only if $X$ is parallel.
  \item A left-invariant vector field $X$ on $G$ is biharmonic (resp. biharmonic map) if and only if $X$ is parallel.
\end{enumerate}
\end{cor}
\begin{cor}\label{co4}
Let $G$ be a three-dimensional unimodular Lie group. Assuming that $\delta_{1}$ and $\delta_{2}$ have the same sign, then a left-invariant vector field $X$ on $G$ is interpolating sesqui-harmonic (resp. interpolating sesqui-harmonic map) if and only if $X$ is parallel.
\end{cor}
\begin{cor}\label{co5}
Let $G$ be a three-dimensional unimodular Lie group. Then
\begin{enumerate}
  \item A left-invariant vector field $X$ on $G$ is harmonic (resp. harmonic map) if and only if $X$ is parallel.
  \item A left-invariant vector field $X$ on $G$ is biharmonic (resp. biharmonic map) if and only if $X$ is parallel.
  \end{enumerate}
\end{cor}
\begin{rem}
We have proved the Corollary \ref{co5} see Theorems 4.1 and 4.2 in \cite{alem}.
\end{rem}

Now, we completely determine the set of left-invariant interpolating sesqui-harmonic vector fields and left-invariant vector fields which are interpolating sesqui-harmonic maps on the Heisenberg group $Nil$ considered as the Cartesian $3$-space $\mathbb{R}^{3}(x,y,z)$ equipped with the left-invariant metric $g_{Nil}$ given by
$g_{Nil}=(dx)^{2}+(dy-xdz)^{2}+(dz)^{2}.$ Since the Heisenberg group $Nil$ is three-dimensional unimodular Lie group, then by the Corollary \ref{co4} we will restrict ourselves to the case where $\delta_{1}$ and $\delta_{2}$ have different signs. The left-invariant vector fields
$\{e_{1}=\frac{\partial}{\partial x},\;e_{2}=\frac{\partial}{\partial y},\;e_{3}=\frac{\partial}{\partial z}+x
    \frac{\partial}{\partial y}\}$
constitute an orthonormal basis of the Lie algebra $\mathfrak{g}$ of $Nil$. The corresponding components of the Levi-Civita connection are determined by
\begin{align}\label{eq3.5}
\nabla_{e_{1}}e_{1}&=0,           &  \nabla_{e_{1}}e_{2} &=-\frac{1}{2}e_{3},              &  \nabla_{e_{1}}e_{3}&=\frac{1}{2}e_{2},\nonumber\\
\nabla_{e_{2}}e_{1}&=-\frac{1}{2}e_{3},         &  \nabla_{e_{2}}e_{2}&=0,   &  \nabla_{e_{2}}e_{3}&=\frac{1}{2}e_{1},\\
\nabla_{e_{3}}e_{1}&=-\frac{1}{2}e_{2},   &  \nabla_{e_{3}}e_{2}&=\frac{1}{2}e_{1},          &  \nabla_{e_{3}}e_{3}&=0.\nonumber
\end{align}
Also the curvature components are given by
\begin{align}\label{eq3.6}
R( e_{1},e_{2})e_{1}&=-\frac{1}{4}e_{2},           &  R( e_{1},e_{2})e_{2} &=\frac{1}{4}e_{1},&  R( e_{1},e_{2})e_{3}&=0,\nonumber\\
R( e_{2},e_{3})e_{1}&=0,         &  R( e_{2},e_{3})e_{2}&=-\frac{1}{4}e_{3},   & R( e_{2},e_{3})e_{3}&=\frac{1}{4}e_{2},\\
R( e_{3},e_{1})e_{1}&=-\frac{3}{4}e_{3},   &  R( e_{3},e_{1})e_{2}&=0,          &  R( e_{3},e_{1})e_{3}&=\frac{3}{4}e_{1}.\nonumber
\end{align}
Let $X=\alpha e_{1}+\beta e_{2}+\gamma e_{3}$ an arbitrary left-invariant vector field on $(Nil,g_{Nil})$. By using \eqref{eq3.5} and \eqref{eq3.6}, we yield
\begin{align}\label{eq3.8}
  \bar{\Delta}X  &= \frac{\alpha}{2}e_{1}+\frac{\beta}{2}e_{2}+\frac{\gamma}{2}e_{3},\nonumber  \\
  \bar{\Delta}\bar{\Delta} X &= \frac{\alpha}{4}e_{1}+\frac{\beta}{4}e_{2}+\frac{\gamma}{4}e_{3},\nonumber \\
  S(X) &= \frac{-\beta \gamma}{4}e_{1}+\frac{\alpha \beta}{4}e_{3}.
\end{align}
By virtue of \eqref{eq3.5}-\eqref{eq3.8}, a long but straightforward calculation we get
\begin{align}\label{eq3.9}
\delta_{1}&\bar{\Delta}X+\delta_{2}\bar{\Delta}\bar{\Delta} X+\delta_{2}\sum_{i=1}^{3}[(\nabla_{e_i}R)(e_i,S(X))X+R(e_i,\nabla_{e_i}S(X))X
\nonumber\\&+2R(e_i,S(X))\nabla_{e_i}X]=\delta_{1}\bar{\Delta}X+\delta_{2}\bar{\Delta}\bar{\Delta} X+\delta_{2}\sum_{i=1}^{3}[\nabla_{e_i}R(e_i,S(X))X\nonumber\\&+R(e_i,S(X))\nabla_{e_i}X]
=\frac{\alpha(8\delta_{1}+\delta_{2}(4+\beta^{2}))}{16}e_{1}\nonumber\\&+
\frac{\beta(8\delta_{1}+\delta_{2}(4+\alpha^{2}+\gamma^{2}))}{16}e_{2}
+\frac{\gamma(8\delta_{1}+\delta_{2}(4+\beta^{2}))}{16}e_{3}.
\end{align}
On the other hand using \eqref{eq3.5}-\eqref{eq3.8}, a long but direct calculations we find
\begin{equation}\label{eq3.10}
    \bar{\Delta}S(X)  = -\frac{\beta\gamma}{8}e_{1}+\frac{\alpha\beta}{8}e_{3},
\end{equation}
\begin{equation}\label{eq3.11}
    R(X,\bar{\Delta}X)S(X) = 0,
\end{equation}
\begin{equation}\label{eq3.12}
    \sum_{i=1}^{3}R(X,\nabla_{e_i}\bar{\Delta}X)e_i = \frac{\beta\gamma}{8}e_{1}-\frac{\alpha\beta}{8}e_{3},
\end{equation}
\begin{equation}\label{eq3.13}
    \sum_{i=1}^{3}R(\nabla_{e_i}X,\bar{\Delta}X)e_i = -\frac{\beta\gamma}{8}e_{1}+\frac{\alpha\beta}{8}e_{3},
\end{equation}
\begin{equation}\label{eq3.14}
    \sum_{i=1}^{3} R(e_i,S(X))e_i = -\frac{\beta\gamma}{8}e_{1}+\frac{\alpha\beta}{8}e_{3},
\end{equation}
\begin{equation}\label{eq3.15}
    \sum_{i=1}^{3} (\nabla_{S(X)}R)(\nabla_{e_i}X,X)e_i = \frac{\beta\gamma(\alpha^{2}+\beta^{2}+\gamma^{2})}{16}e_{1}-
    \frac{\alpha\beta(\alpha^{2}+\beta^{2}+\gamma^{2})}{16}e_{3},
\end{equation}
\begin{equation}\label{eq3.16}
    \sum_{i=1}^{3}R(X,\nabla_{e_i}X)\nabla_{e_i}S(X) =
    \frac{\beta\gamma(3\alpha^{2}+3\gamma^{2}-\beta^{2})}{64}e_{1}-\frac{
    \alpha\beta(3\alpha^{2}+3\gamma^{2}-\beta^{2})}{64}e_{3},
\end{equation}
\begin{equation}\label{eq3.17}
    \sum_{i=1}^{3}R(X,R(e_i,S(X))X)e_i =  -\frac{\beta\gamma(9\alpha^{2}+9\gamma^{2}+\beta^{2})}{64}e_{1}+\frac{
    \alpha\beta(9\alpha^{2}+9\gamma^{2}+\beta^{2})}{64}e_{3}.
\end{equation}
From \eqref{eq3.10}-\eqref{eq3.17}, we yield
\begin{align}\label{eq3.18}
   &\delta_{1}S(X)+\delta_{2}\bar{\Delta}S(X)+\delta_{2}R(X,\bar{\Delta} X)S(X)-\delta_{2}\sum_{i=1}^{3}[R(X,\nabla_{e_i}\bar{\Delta} X)e_i\nonumber\\&
   -R(\nabla_{e_i}X,\bar{\Delta} X)e_i-R(e_i,S(X))e_i
   -(\nabla_{S(X)}R)(\nabla_{e_i}X,X)e_i\nonumber\\
   &+R(X,\nabla_{e_i}X)\nabla_{e_i}S(X)-
   R(X,R(e_i,S(X))X)e_i]\nonumber\\&=-\frac{\beta \gamma(4\delta_{1}+\delta_{2}(8+\alpha^{2}+\gamma^{2}-2\beta^{2}))}{16}e_{1}+
   \frac{\alpha \beta(4\delta_{1}+\delta_{2}(8+\alpha^{2}+\gamma^{2}-2\beta^{2}))}{16}e_{3}.
\end{align}
From \eqref{eq3.9} and \eqref{eq3.18}, we deduce that $X$ is an interpolating sesqui-harmonic map if and only if
\begin{equation}\label{eq3.19}
\left\{
  \begin{array}{ll}
    \alpha(8\delta_{1}+\delta_{2}(4+\beta^{2}))=0, & \hbox{} \\
    \beta(8\delta_{1}+\delta_{2}(4+\alpha^{2}+\gamma^{2}))=0, & \hbox{} \\
    \gamma(8\delta_{1}+\delta_{2}(4+\beta^{2}))=0, & \hbox{}
  \end{array}
\right.
\end{equation}
and
\begin{equation}\label{eq3.20}
\left\{
  \begin{array}{ll}
    \beta \gamma(4\delta_{1}+\delta_{2}(8+\alpha^{2}+\gamma^{2}-2\beta^{2}))=0, & \hbox{} \\
    \alpha \beta(4\delta_{1}+\delta_{2}(8+\alpha^{2}+\gamma^{2}-2\beta^{2}))=0. & \hbox{}
  \end{array}
\right.
\end{equation}
In particular, $X$ is an interpolating sesqui-harmonic vector field if and only if \eqref{eq3.19} holds. From
the system \eqref{eq3.19}, the subcases $\beta=\gamma=0$, $\alpha=\gamma=0$  and $\alpha=\beta=0$ give vector fields $X=\alpha e_{1}$, $X=\beta e_{2}$ and $X=\gamma e_{3}$ respectively such that $\delta_{1}=-\frac{1}{2}\delta_{2}$ which define by \eqref{eq3.20} an interpolating sesqui-harmonic maps. If $\alpha=0$ (resp. $\gamma=0$), the system \eqref{eq3.19} give the solutions $X=\pm2\sqrt{-\frac{(2\delta_{1}+\delta_{2})}{\delta_{2}}}e_{2}
\pm2\sqrt{-\frac{(2\delta_{1}+\delta_{2})}{\delta_{2}}}e_{3}$ (resp. $X=\pm2\sqrt{-\frac{(2\delta_{1}+\delta_{2})}{\delta_{2}}}e_{1}
\pm2\sqrt{-\frac{(2\delta_{1}+\delta_{2})}{\delta_{2}}}e_{2}$) whenever $\frac{(2\delta_{1}+\delta_{2})}{\delta_{2}}<0$, which define by \eqref{eq3.20} an interpolating sesqui-harmonic maps if $\delta_{1}=-\delta_{2}$. If $\beta=0$, the system \eqref{eq3.19} gives the solution $X=\alpha e_{1}+
\gamma e_{3}$ such that $\delta_{1}=-\frac{1}{2}\delta_{2}$ which defines by \eqref{eq3.20} an interpolating sesqui-harmonic map. If $\alpha\neq 0$, $\beta\neq 0$, $\gamma\neq 0$, the system \eqref{eq3.19} admits $\beta=\pm2\sqrt{-\frac{(2\delta_{1}+\delta_{2})}{\delta_{2}}}$ and the cylinder
$\alpha^{2}+\gamma^{2}=-4\frac{(2\delta_{1}+\delta_{2})}{\delta_{2}}$ whenever $\frac{(2\delta_{1}+\delta_{2})}{\delta_{2}}<0$. Therefore, the coordinates of $X$ satisfy the equations of circles:
\begin{align*}
   C^{1}&=\biggl\{\alpha^{2}+\gamma^{2}=-4\frac{(2\delta_{1}+\delta_{2})}{\delta_{2}},\;
   \beta=2\sqrt{-\frac{(2\delta_{1}+\delta_{2})}{\delta_{2}}}\biggr\};\\
   C^{2}&=\biggl\{\alpha^{2}+\gamma^{2}=-4\frac{(2\delta_{1}+\delta_{2})}{\delta_{2}},\;
   \beta=-2\sqrt{-\frac{(2\delta_{1}+\delta_{2})}{\delta_{2}}}\biggr\},
\end{align*}
and by virtue of \eqref{eq3.20} the corresponding vector fields define an interpolating sesqui-harmonic maps if $\delta_{1}=-\delta_{2}$.

Summarizing, we yield
\begin{thm}\label{Th7}
Let $X=\alpha e_{1}+\beta e_{2}+\gamma e_{3}$ be a left-invariant vector field on the Heisenberg group $(Nil,g_{Nil})$. Then
$X$ is an interpolating sesqui-harmonic vector field if and only if
\begin{itemize}
  \item $X=\alpha e_{1},\;\beta e_{2},\;\gamma e_{3},$ with $\delta_{1}=-\frac{1}{2}\delta_{2}$. The corresponding vector fields define an interpolating sesqui-harmonic maps.
  \item $X=\pm2\sqrt{-\frac{(2\delta_{1}+\delta_{2})}{\delta_{2}}}e_{2}
\pm2\sqrt{-\frac{(2\delta_{1}+\delta_{2})}{\delta_{2}}}e_{3}$ (resp. $X=\pm2\sqrt{-\frac{(2\delta_{1}+\delta_{2})}{\delta_{2}}}e_{1}
\pm2\sqrt{-\frac{(2\delta_{1}+\delta_{2})}{\delta_{2}}}e_{2}$) whenever $\frac{(2\delta_{1}+\delta_{2})}{\delta_{2}}<0$, which define an interpolating sesqui-harmonic maps if $\delta_{1}=-\delta_{2}$.
  \item $X=\alpha e_{1}+\beta e_{2}+\gamma e_{3}$ such that the coordinates of $X$ satisfy the equations of circles $C^{1}$ and $C^{2}$. The corresponding vector fields define an interpolating sesqui-harmonic maps if $\delta_{1}=-\delta_{2}$.
\end{itemize}
\end{thm}
\begin{rem}It should be noted already that a vector field is interpolating sesqui-harmonic does not automatically imply that the corresponding map is interpolating sesqui-harmonic map. Indeed, from Theorem \ref{Th7} if $\delta_{1}\neq-\delta_{2}$, the vector fields  $X=\pm2\sqrt{-\frac{(2\delta_{1}+\delta_{2})}{\delta_{2}}}e_{1}
\pm2\sqrt{-\frac{(2\delta_{1}+\delta_{2})}{\delta_{2}}}e_{2}$ are interpolating sesqui-harmonic but do not define an interpolating sesqui-harmonic maps.
\end{rem}

\end{document}